\documentclass[12pt,reqno]{amsart}

\usepackage{multirow}
\usepackage{hhline}

\usepackage{mathtools}
\usepackage{times}
\usepackage[T1]{fontenc}
\usepackage{mathrsfs}
\usepackage{latexsym}
\usepackage[dvips]{graphics}
\usepackage[titletoc, title]{appendix}
\usepackage{amsmath,amsfonts,amsthm,amssymb,amscd,float}
\usepackage[dvipsnames]{xcolor}
\usepackage{hyperref}
\usepackage{amsmath}

\usepackage{color}
\usepackage{breakurl}

\usepackage{comment}
\newcommand{\bburl}[1]{\textcolor{blue}{\url{#1}}}
\newcommand{\seqnum}[1]{\href{https://oeis.org/#1}{\rm \underline{#1}}}

\newtheorem{thm}{Theorem}[section]

\newtheorem{cor}[thm]{Corollary}
\newtheorem{claim}[thm]{Claim}
\newtheorem{lem}[thm]{Lemma}
\newtheorem{prop}[thm]{Proposition}
\newtheorem{exa}[thm]{Example}

\newtheorem{defi}[thm]{Definition}
\newtheorem{rek}[thm]{Remark}

\numberwithin{equation}{section}

\DeclareFontFamily{U}{mathx}{}
\DeclareFontShape{U}{mathx}{m}{n}{<-> mathx10}{}
\DeclareSymbolFont{mathx}{U}{mathx}{m}{n}
\DeclareMathAccent{\widehat}{0}{mathx}{"70}
\DeclareMathAccent{\widecheck}{0}{mathx}{"71}

\begin{document}

\title[The Best Two-term Underapproximation by Egyptian Fractions]{A Threshold for the Best Two-term Underapproximation by Egyptian Fractions}

\author{H\`ung Vi\d{\^e}t Chu}

\email{\textcolor{blue}{\href{mailto:hungchu1@tamu.edu}{hungchu1@tamu.edu}}}
\address{Department of Mathematics, Texas A\&M University, College Station, TX 77843, USA}

\begin{abstract} 
Let $\mathcal{G}$ be the greedy algorithm that, for each $\theta\in (0,1]$, produces an infinite sequence of positive integers $(a_n)_{n=1}^\infty$  satisfying $\sum_{n=1}^\infty 1/a_n = \theta$. For natural numbers $p < q$, let $\Upsilon(p,q)$ denote the smallest positive integer $j$ such that $p$ divides $q+j$. Continuing Nathanson's study of two-term underapproximations, we show that whenever $\Upsilon(p,q) \leqslant 3$, $\mathcal{G}$ gives the (unique) best two-term underapproximation of $p/q$;
i.e., if $1/x_1 + 1/x_2 < p/q$ for some $x_1, x_2\in \mathbb{N}$, then $1/x_1 + 1/x_2 \leqslant 1/a_1+1/a_2$. However, the same conclusion fails for every $\Upsilon(p,q)\geqslant 4$. 

Next, we study stepwise underapproximation by $\mathcal{G}$. Let $e_{m} = \theta - \sum_{n=1}^{m}1/a_n$ be the $m$\textsuperscript{th} error term. We compare $1/a_m$ to a superior underapproximation of $e_{m-1}$, denoted by $N/b_m$ ($N \in\mathbb{N}_{\geqslant 2}$), and characterize when  $1/a_m = N/b_m$. One characterization is 
$a_{m+1}  \geqslant N a_m^2 - a_m + 1$. 
Hence, for rational $\theta$, we only have $1/a_m = N/b_m$ for finitely many $m$. However, there are irrational numbers such that $1/a_m = N/b_m$ for all $m$. Along the way, various auxiliary results are encountered. 
\end{abstract}



\thanks{The author would like to thank the anonymous referee for a careful reading and useful suggestions that help improve the paper's expositions.}

\maketitle

\tableofcontents

\section{Introduction and main results}

An \textit{Egyptian} fraction is a fraction of the form $1/n$ for $n\in\mathbb{N}$. There has been extensive literature on representations of numbers by a finite sum of Egyptian fractions: \cite{AB15, BE22, CEJ12, Cu22, El16, EP21, Ep21, HV13, Ke21, LS22, Ma99}, to name a few. The present note is concerned with representing a number $\theta\in (0,1]$ as an infinite sum of Egyptian fractions. We use the infinite greedy algorithm $\mathcal{G}$ (recently examined in \cite{Ch23, Na23}), which at each step, chooses the largest Egyptian fraction so that the sum of all chosen fractions up to that step is strictly smaller than $\theta$. For example, the corresponding series for $\theta = 1/7$ is 
$$\frac{1}{8} + \frac{1}{57} + \frac{1}{3193} + \frac{1}{10192057} + \frac{1}{103878015699193} + \cdots.$$
We can describe $\mathcal{G}$ more concretely using the function $G:(0,1]\rightarrow\mathbb{N}$
$$G(\theta)\ :=\ \left\lfloor \frac{1}{\theta}\right\rfloor+1.$$
It is readily checked that $1/G(\theta)$ is the largest Egyptian fraction smaller than $\theta$; i.e.,
\begin{equation}\label{e13}\frac{1}{G(\theta)}\ <\ \theta\ \leqslant\ \frac{1}{G(\theta)-1}.\end{equation}
Hence, the denominators of fractions obtained from $\mathcal{G}$ are given recursively as
\begin{equation}\label{e-1}a_1\ :=\ G(\theta)\mbox{ and }a_m\ :=\ G\left(\theta - \sum_{n=1}^{m-1}\frac{1}{a_n}\right), \quad m\geqslant 2.\end{equation}
We denote the sequence $(a_n)_{n=1}^\infty$ above by $\mathcal{G}(\theta)$, which grows quickly. Particularly, 
consecutive terms of $(a_n)_{n=1}^\infty$ satisfy the following neat inequality (see \cite[(3)]{Na23})
 \begin{equation}\label{e1}a_1\geqslant 2\mbox{ and } a_{n+1}\ \geqslant\ a_n^2-a_n+1, \quad n\geqslant 1.\end{equation}
From the definition of $G$, \eqref{e-1}, and \eqref{e1}, we know that
$$0\ < \ \theta - \sum_{n=1}^{m-1} \frac{1}{a_n}\ \leqslant\ \frac{1}{a_{m}-1}\ \rightarrow\ 0;$$
thus, $\sum_{n= 1}^{\infty}1/a_n \ =\ \theta$, as desired. 
The sequence $\mathcal{G}(1)$ is the well-known Sylvester's sequence (\seqnum{A000058} in \cite{Sl23}; see also \cite{Na23, ES64, So05})
$$2, 3, 7, 43, 1807, 3263443, 10650056950807, 113423713055421844361000443,\ldots.$$
It is worth mentioning that the Sylvester sequence is minimal in the sense that the sequence satisfies \eqref{e1} with the smallest possible value of $a_1 (= 2)$ and with the equality $a_{n+1} = a_n^2 - a_n +1$. By induction, we also see that $\mathcal{G}(1) = (a_n)_{n=1}^\infty$ satisfies the recurrence $a_{n+1} = a_1a_2\cdots a_n+1$ for all $n\geqslant 1$ (see \cite[Theorem 1]{Na23}.)

For $\theta\in (0,1]$ and $\mathcal{G}(\theta) = (a_n)_{n=1}^\infty$, let $\sum_{n=1}^m 1/a_n$ be the $m$-term underapproximation by $\mathcal{G}$. 
Then $\sum_{n=1}^m 1/a_n$ is said to be the best $m$-term underapproximation of $\theta$ if for any integers $(x_n)_{n=1}^\infty\subset \mathbb{N}$ with $\sum_{n=1}^m 1/x_n < \theta$, we have 
$$\sum_{n=1}^m \frac{1}{x_n} \ \leqslant\ \sum_{n=1}^m \frac{1}{a_n} \ <\ \theta.$$
Due to the definition of $\mathcal{G}$, $1/a_1$ is the best one-term underapproximation. 
However, $\mathcal{G}$ may not give the best two-term underapproximation as the following example indicates. 

\begin{exa}\normalfont
Take $\theta = 5/16$. The two-term underapproximation by $\mathcal{G}$ is $1/4 + 1/17$. Observe that 
$$\frac{1}{4} + \frac{1}{17} \ <\ \frac{1}{5} + \frac{1}{9} \ <\ \frac{5}{16}.$$
\end{exa}

Nathanson \cite[Sections 6 and 7]{Na23} studied the best two-term underapproximation and gave a method to determine whether $\mathcal{G}$ gives the (unique) best  two-term underapproximation for any given $\theta\in (0,1]$. This method reduces the problem to checking a finite number of cases by computer, which he then demonstrated  by considering $\theta$ with small $\lfloor 1/\theta\rfloor\in \{1,2\}$ (\cite[Theorems 6 and 7]{Na23}). While this method is very neat and thorough, its limitation lies in the increasing number of cases to be checked as $\lfloor 1/\theta\rfloor$ increases. Therefore, it is a natural problem to construct a family of $\theta$ with arbitrarily large $\lfloor 1/\theta\rfloor$ (thus, applying Nathanson's method would not be feasible) such that $\mathcal{G}$ always gives the best two-term underapproximation for $\theta$ in the family. From \cite[Theorem 5]{Na23} (see Theorem \ref{Nathm} below), we know one such family 
$$\mathcal{F} = \{p/q\,:\, p, q\in \mathbb{N}, p < q, p \mid (q+1)\}.$$
We introduce the function $\Upsilon(p,q)$ that captures the divisibility between the numerator and denominator of fractions in $\mathcal{F}$.

\begin{defi}\normalfont
For $p, q\in \mathbb{N}$, let $\Upsilon(p, q)$ be the smallest positive integer $m$ such that $p$ divides $q+m$.
\end{defi}
Numbers $p/q\in \mathcal{F}$ have $\Upsilon(p,q) = 1$. 
The main ingredient for the proof of \cite[Theorem 5]{Na23} is a corollary of Muirhead’s inequality \cite{Na22}, which was directly proved by Ambro and Barc\v{a}u \cite{AB15}. The proof also takes advantage of the fact that $\Upsilon(p,q) = 1$. However, as our first result gives a new family with $\Upsilon(p,q) \in \{2,3\}$, we devise a different argument, which involves proving certain Diophantine inequalities (see Lemmas \ref{lp1} and \ref{lp11}). We state our first result.

\begin{thm}\label{mk0}
Let $p, q\in \mathbb{N}$ with $p < q$. If $\Upsilon(p,q) \leqslant 3$ and $p/q\neq 10/17$, then $\mathcal{G}$ gives the unique best two-term underapproximation of $p/q$. 

When $p/q = 10/17$, $\mathcal{G}$ still gives the best (not unique) two-term underapproximation. In particular, besides the underapproximation $1/2+1/12$ given by the greedy algorithm, there is exactly one other equal underapproximation, namely $1/3+1/4$. 

Furthermore, for each $k\geqslant 4$, there exist $p < q$ such that $\Upsilon(p,q)= k$, and $\mathcal{G}$ does not give the best two-term underapproximation. 
\end{thm}

\begin{rek}\normalfont
The case $\Upsilon(p,q) = 1$ in Theorem \ref{mk0} follows from \cite[Theorem 5]{Na23}.
\end{rek}

For our second result, let us take a closer look at each step of $\mathcal{G}$. For $m\geqslant 1$, $\mathcal{G}$ approximates the $(m-1)$\textsuperscript{th} error term 
$$e_{m-1}\ :=\ \theta - \sum_{n=1}^{m-1} \frac{1}{a_n}.$$
by $1/a_m$, the largest fraction with numerator $1$ and smaller than $e_{m-1}$. The restriction here is the numerator $1$, because we are using Egyptian fractions. This makes us wonder about a superior choice to approximate $e_{m-1}$ from below, where the numerator is $N$ for some $N\geqslant 2$. Denote the largest fraction of numerator $N$ and smaller than $e_{m-1}$ by $N/ b_m$ $(b_m\in \mathbb{N})$. This choice is indeed superior as by the definition of $N/b_m$, 
$$\frac{1}{a_m}\ =\ \frac{N}{Na_m}\ \leqslant\ \frac{N}{b_m} \ <\ e_{m-1}.$$
It is possible that $1/a_m < N/b_m$. For example, taking $\theta = 1$, $m = 1$, and $N = 2$, we easily compute that 
$$\frac{1}{a_1} \ =\ \frac{1}{2} \ <\ \frac{2}{3} \ =\ \frac{2}{b_1}\ < \ e_0 \ =\ 1.$$
When are the two choices of underapproximation, $1/a_m$ and $N/b_m$, the same? 
We shall give several characterizations of such situation, one of which states that $1/a_m = N/b_m$ if and only if 
\begin{equation}\label{f1} a_{m+1} \ \geqslant\ N a_m^2-a_m+1,\end{equation}
a generalization of \eqref{e1}. 
As we shall see, if  $N\geqslant 2$ and $\theta$ is rational, the above inequality fails for large $m$; therefore, fractions by $\mathcal{G}$ are eventually strictly dominated by $N/b_m$. On the other hand, there are irrational numbers (and only irrational numbers) such that $1/a_m = N/b_m$ for every $m$. It is interesting to see that $\mathcal{G}$ behaves differently between the rational and the irrational. Before stating our second theorem, we introduce the function $\Phi: (0,1]\rightarrow [1, \infty)$ (that appears in one of the characterizations)
$$\Phi(\theta)\ :=\ \frac{1}{1-\left\{\theta^{-1}\right\}}\ =\ \frac{1}{G(\theta)-1/\theta}, \quad \theta\in (0,1],$$ where $\{x\}: = x - \lfloor x\rfloor$ for any real $x$.
\begin{figure}[H]
\centering
\includegraphics[scale=1.15]{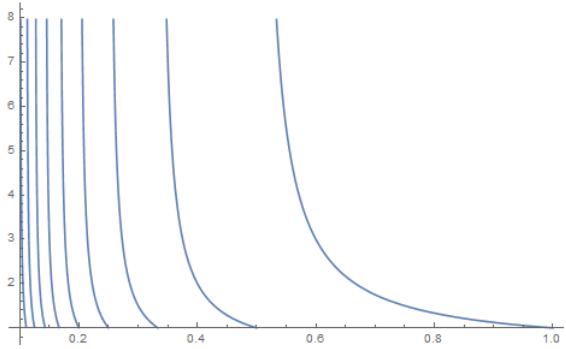}
\caption{Graph of $\Phi(x)$ for $x\in [1/10,1]$. The set of discontinuities of $\Phi$ is $D = \{1/n: n\in \mathbb{N}_{\geqslant 2}\}$. In particular, for $n\in\mathbb{N}_{\geqslant 2}$, $\lim_{\theta\rightarrow (1/n)^+}\Phi(\theta) = +\infty$, while $\lim_{\theta\rightarrow (1/n)^-}\Phi(\theta) \ =\ 1^+$.}
\end{figure}

\begin{thm}\label{mk}
Fix $N\in \mathbb{N}$. Let $\theta\in (0,1]$, $(a_n)_{n=1}^\infty = \mathcal{G}(\theta)$, and $e_{m} = \theta - \sum_{n=1}^m 1/a_n$. Let $N/b_m$ be the largest fraction strictly smaller than $e_{m-1}$. The following are equivalent:
\begin{enumerate}
\item[i)] $a_{m+1} \geqslant Na_m^2 - a_m + 1$, 
\item[ii)] $1/a_m = N/b_m$,
\item[iii)] $\Phi(e_{m-1}) \geqslant N$,
\item[iv)] $e_{m-1}\leqslant N/(N a_m - 1)$.
\end{enumerate}
\end{thm}

The following proposition, combined with Theorem \ref{mk}, shows that if $\theta$ is rational, $1/a_m < N/b_m$ for large $m$.

\begin{prop}\label{mk2}
Let $\theta = p/q\in (0,1]$, where $p, q\in \mathbb{N}$ and $\gcd(p,q) = 1$. Let $(a_n)_{n=1}^\infty = \mathcal{G}(\theta)$. We have
$$\left(\theta - \sum_{n=1}^{\ell}\frac{1}{a_n}\right)^{-1}\ \in\ \mathbb{N},$$
where $\ell$ is the smallest nonnegative integer such that $p$ divides $q+\ell$. As a result, 
\begin{equation}\label{e-30}a_{n+1}\ =\ a_n^2 - a_n + 1, \quad n\geqslant \ell+1.\end{equation}
\end{prop}

To avoid repetition, for the next corollaries, let $\theta\in (0,1]$, $(a_n)_{n=1}^\infty = \mathcal{G}(\theta)$, and $e_{m} = \theta - \sum_{n=1}^m 1/a_n$. Fix $N\geqslant 2$ and let $N/b_m$ be the largest fraction strictly smaller than $e_{m-1}$. 

\begin{cor}
If $\theta$ is rational, then for sufficiently large $m$, 
$$\frac{1}{a_m} \ <\ \frac{N}{b_m}.$$
\end{cor}
\begin{proof}
Use Theorem \ref{mk} and Proposition \ref{mk2}.
\end{proof}

\begin{cor} The set of all $\theta\in (0,1]$ such that $1/a_m = N/b_m$ for all $m$ is precisely 
$$\Gamma \ =\ \left\{\sum_{n=1}^\infty \frac{1}{c_n}\,:\, c_1\geqslant 2 \mbox{ and } c_{n+1}\geqslant N c_n^2 - c_n + 1, \quad n\geqslant 1\right\}.$$

Furthermore, $\Gamma$ contains only irrational numbers.
\end{cor}
\begin{proof}
The first statement is immediate from Theorem \ref{mk} and \cite[Corollary 3]{Na23}, which implies that each sequence $(c_n)_{n=1}^\infty$ in the definition of $\Gamma$ is in fact $\mathcal{G}(\sum_{n=1}^\infty 1/c_n)$.
The second statement follows from Proposition \ref{mk2}.
\end{proof}

Our next result generalizes \cite[Theorem 1]{Na23}.

\begin{thm}\label{mk4}
Let $\theta = p/q\in (0,1]$, where $p, q\in \mathbb{N}$. Let $\mathcal{G}(\theta) = (a_n)_{n=1}^\infty$. If $\Upsilon(p,q)$ divides $q$, then 
$$a_1\ =\ \frac{q+\Upsilon(p,q)}{p}, \mbox{ and }$$
$$\frac{p}{q}-\sum_{i=1}^{n-1}\frac{1}{a_i}\ =\ \frac{1}{\frac{q}{\Upsilon(p,q)}\prod_{i=1}^{n-1}a_i}, \quad n\geqslant 2.$$
Hence,
\begin{equation}\label{ef1}a_n \ =\ \frac{q}{\Upsilon(p,q)}\prod_{i=1}^{n-1}a_i + 1, \quad n\geqslant 2.\end{equation}
\end{thm}

We also know a nice recurrence when $\Upsilon(p,q) = 2$ and $q$ is odd.

\begin{prop}\label{pl1}
Let $\theta = p/q\in (0,1]$, where $p, q\in \mathbb{N}$, $q$ is odd, and $\Upsilon(p,q) = 2$. Let $\mathcal{G}(\theta) = (a_n)_{n=1}^\infty$. Then 
\begin{equation*}a_1\ =\ \frac{q+2}{p}, \label{ef2}a_2\ =\ \left\lfloor \frac{qa_1}{2}\right\rfloor+1, \mbox{ and } \frac{p}{q}-\sum_{i=1}^{n-1}\frac{1}{a_i}\ =\ \frac{1}{q\prod_{i=1}^{n-1}a_i}, \quad n\geqslant 3.\end{equation*}
Hence,
\begin{equation*}a_n \ =\ q\prod_{i=1}^{n-1}a_i + 1, \quad n\geqslant 3.\end{equation*}
\end{prop}

It would be interesting to see explicit formulas for $\mathcal{G}(\theta)$ for larger families of $\theta$ when $\Phi(p,q)$ does not divide $q$. Shifting away from the best $2$-term and stepwise underapproximation, we end this paper by discussing the best $m$-term underapproximation by $\mathcal{G}$. Let us recall a recent theorem of Nathanson.
\begin{thm}\cite[Theorem 5]{Na23}\label{Nathm}
Let $p<q\in \mathbb{N}$ such that $p$ divides $q+1$, and let $(a_n)_{n=1}^\infty = \mathcal{G}(p/q)$. For every $m\in \mathbb{N}$, $\sum_{n=1}^m 1/a_n$ is the unique best $m$-term underapproximation of $p/q$. 
\end{thm}
Nathanson \cite{Na23} asked ``Do other rational numbers have this property?". As an application of Proposition \ref{pl1}, we answer this question positively. 
\begin{thm}\label{fthm2}
Let $p < q\in \mathbb{N}$ with odd $q$ and $\Upsilon(p,q) = 2$. Let $(a_n)_{n=1}^\infty = \mathcal{G}(p/q)$. For every $m\in \mathbb{N}$, $\sum_{n=1}^m 1/a_n$ is the unique best $m$-term underapproximation of $p/q$. 
\end{thm}

Our present paper is structured as follows: Section \ref{best2} proves the first statement of Theorem \ref{mk0}, while Section \ref{>3} proves the third statement. (The second statement of Theorem \ref{mk0} is already contained in \cite[Theorem 6]{Na23}.) Section \ref{compare} proves Theorem \ref{mk} and Proposition \ref{mk2} along with its optimality. Finally, Section \ref{gen} proves Theorem \ref{mk4}, Proposition \ref{pl1}, and Theorem \ref{fthm2}. 


\section{The best two-term underapproximation when $\Upsilon(p,q)\leqslant 3$}\label{best2}
We shall establish the first statement in Theorem \ref{mk0}. Throughout this section, for a given $\theta\in (0,1]$, let $a_1, a_2$ be the first two terms in $\mathcal{G}(\theta)$ and let $2\leqslant x_1 \leqslant x_2$ be integers satisfying
\begin{equation}\label{ep8}\frac{1}{a_1} + \frac{1}{a_2}\ \leqslant\ \frac{1}{x_1} + \frac{1}{x_2} \ <\ \theta.\end{equation}

\subsection{Preliminaries}
We first recall \cite[Lemma 1]{Na23}, which gives us a necessary condition on when $\mathcal{G}$ does not give the unique best two-term underapproximation. 

\begin{lem}\cite[Lemma 1]{Na23}\label{l1Na23}
Let $a_1$ and $a_2$ be integers such that 
$$a_1\ \geqslant\ 2\mbox{ and }a_2\ \geqslant\ a_1(a_1-1)+1.$$
If $x_1$ and $x_2$ are integers such that
$$2\ \leqslant\ x_1\ \leqslant\ x_2\mbox{ and }(x_1,x_2)\ \neq\ (a_1, a_2)$$
and
$$\frac{1}{a_1}+\frac{1}{a_2}\ \leqslant\ \frac{1}{x_1} + \frac{1}{x_2}\ <\ \frac{1}{a_1}+\frac{1}{a_2-1}$$
then
$$a_1 + 1\ \leqslant\ x_1\ \leqslant\ 2a_1-1\ \leqslant\ x_2\ <\ \frac{a_1x_1}{x_1-a_1}$$
and 
$$x_2\ \leqslant\ a_2-1.$$
\end{lem}

\begin{lem}\label{lp2}
If \eqref{ep8} holds for some $(x_1, x_2)\neq (a_1, a_2)$, then
$$a_1 + 1\ \leqslant\ x_1 \ \leqslant\ 2a_1-1.$$
\end{lem}

\begin{proof}
Since $a_1$ and $a_2$ are the first two terms obtained from the greedy algorithm $\mathcal{G}$ applied to $\theta$, 
$$\frac{1}{a_1} + \frac{1}{a_2} \ <\ \theta \ \leqslant\ \frac{1}{a_1} + \frac{1}{a_2-1}.$$
We now use Lemma \ref{l1Na23} to reach the desired conclusion.
\end{proof}

\begin{lem}\label{lp3}
Assume that \eqref{ep8} holds for some $(x_1,x_2)\neq (a_1,a_2)$. Let $x_2' = G(\theta-1/x_1)$. We have
\begin{equation}\label{ep9}x'_2\ \leqslant\ \left(\frac{1}{a_1} + \frac{1}{a_2} - \frac{1}{x_1}\right)^{-1}.\end{equation}
\end{lem}
\begin{proof}
By definition, $1/x_2'$ is the largest Egyptian fraction less than $\theta - 1/x_1$. Hence, 
$$\frac{1}{a_1} + \frac{1}{a_2} \ \leqslant\ \frac{1}{x_1} + \frac{1}{x_2} \ \leqslant\ \frac{1}{x_1} + \frac{1}{x_2'},$$
giving
$$\frac{1}{x_2'}\ \geqslant\ \frac{1}{a_1} + \frac{1}{a_2} - \frac{1}{x_1}.$$
By Lemma \ref{lp2}, 
$$\frac{1}{a_1} + \frac{1}{a_2} - \frac{1}{x_1} \ >\ 0.$$
We, therefore, obtain
$$x'_2\ \leqslant\ \left(\frac{1}{a_1} + \frac{1}{a_2} - \frac{1}{x_1}\right)^{-1}.$$
\end{proof}

We shall prove the first statement of Theorem \ref{mk0} by contradiction. We argue that for $\theta$ in our family, if \eqref{ep8} holds for some $(x_1, x_2)\neq (a_1, a_2)$, then \eqref{ep9} fails, contradicting Lemma \ref{lp3}. Again, we consider only the cases $\Upsilon(p,q) = 2$ and $\Upsilon(p,q) = 3$. The case $\Upsilon(p,q) = 1$ follows from Theorem \ref{Nathm}.

\subsection{When $\Upsilon(p, q) = 2$}
The next technical lemma proves a key inequality.

\begin{lem}\label{lp1}
Let $q \geqslant 4$, $u\geqslant 2$, $v\in \{1,2\}$, and $1\leqslant s\leqslant u-1$. If $(q+2)/u\in \mathbb{N}_{\geqslant 3}$, then 
\begin{equation}\label{ep3}\left\lfloor \frac{qu(u+s)}{s(q+2)+2u}\right\rfloor \ >\ \frac{(qu+v)u(u+s)}{squ+vs+2u(u+s)}-1.\end{equation}
\end{lem}

\begin{proof}
Fix $u\geqslant 2$. The smallest $q$ such that $(q+2)/u\in \mathbb{N}_{\geqslant 3}$ is equal to $3u-2$. We need to confirm \eqref{ep3} for all $q$ in the set 
$\left\{ku-2\,:\, k\geqslant 3\right\}$. Choose $k\geqslant 3$ and replace $q = ku-2$ in \eqref{ep3}. The left side of \eqref{ep3} becomes 
$$L\ :=\ \left\lfloor \frac{(ku-2)u(u+s)}{kus+2u}\right\rfloor\ =\ \left\lfloor u+\frac{ku^2-4u-2s}{ks+2}\right\rfloor.$$
The right side of \eqref{ep3} becomes
$$R\ := \ u + \frac{ku^4-4u^3-2u^2s+vu^2}{ku^2s+vs+2u^2}-1\ =\ u +\frac{ku^2-4u-2s+v}{ks+2+vs/u^2}-1.$$
Observe that $ku^2-4u-2s > 0$ because 
$$ku^2\ \geqslant\ 3u^2\ \geqslant\ 6u\ >\ 6u-2 \ =\ 4u+2(u-1)\ \geqslant\ 4u+2s.$$
To establish $L > R$, we first get rid of the floor function by writing 
\begin{equation}\label{ep5}ku^2-4u-2s \ =\ (ks+2)\ell+j,\end{equation}
for some $\ell\geqslant 0$ and $0\leqslant j\leqslant ks+1$. 
Then $$L \ =\  u + \ell$$
and 
$$R\ =\ u+\ell + \frac{j+v-vs\ell/u^2}{ks+2+vs/u^2}-1.$$
Hence,
\begin{equation}\label{ep2}L- R\ =\ \frac{(ks-j)u^2+(2-v)u^2+vs+vs\ell}{(ks+2)u^2+vs}.\end{equation}
Our goal is to show that the right side of \eqref{ep2} is positive. We proceed by case analysis.

If $0\leqslant j\leqslant ks$, then $L-R > 0$. 

If $j = ks+1$, then from \eqref{ep2},
$$L - R\ =\ \frac{(1-v)u^2+vs+vs\ell}{(ks+2)u^2+vs}.$$ 
If $v = 1$, then $L-R > 0$. For the rest of the proof, assume $v = 2$. 
We want to show that $2s(\ell+1) > u^2$. Suppose, for a contradiction, that $2s(\ell+1) \leqslant u^2$. We shall establish finite ranges for $u, s,$ and $k$ then resort to a brute-force check to obtain a contradiction. 

By \eqref{ep5}, 
\begin{align*}
2ku^2&\ =\ 2(ks+2)\ell + 2ks + 2 + 8u + 4s\\
&\ \leqslant\ k(u^2-2s) + 4\ell + 2ks+ 2 + 8u+4(u-1)\\
&\ \leqslant\ ku^2+2u^2/s+12u-6\ \leqslant\ (2+k)u^2 + 12u-6.
\end{align*}
Hence, 
\begin{equation}\label{ep7} u^2\ \leqslant\ (k-2)u^2\ \leqslant\ 12u-6,\end{equation}
giving $2\leqslant u\leqslant 11$. 

Note that \eqref{ep5} and $j = ks+1$ imply that $k$ must be odd. If $k \geqslant 7$, then \eqref{ep7} gives
$5u^2 \leqslant 12u-6$, which contradicts that $u\geqslant 2$. Hence, either $k = 3$ or $k = 5$. 

Case 1: $k = 3$. By \eqref{ep5}, 
$$3u^2 \ =\ (3s+2)\ell + 5s + 4u+1.$$
For each choice of $2\leqslant u\leqslant 11$, we have $1\leqslant s\leqslant u-1$. A brute-force check shows that only $(u,s) = (7,4)$ gives an integer value for $\ell$. In that case, $\ell = 7$. However, these values contradict our supposition $2s(\ell+1)\leqslant u^2$.

Case 2: $k = 5$. By \eqref{ep5}, 
$$5u^2\ =\ (5s+2)\ell+7s+4u+1.$$
For each choice of $2\leqslant u\leqslant 11$ and $1\leqslant s\leqslant u-1$, a brute-force check shows that only $(u,s) \in \{(4,1), (8,1), (11,1), (11,8)\}$ gives integer values for $\ell$. However, all these choices contradict $2s(\ell+1)\leqslant u^2$. 

We conclude that $L - R > 0$.
\end{proof}

\begin{thm}\label{mk01}
Let $p, q\in \mathbb{N}$ with $p < q$ and $\Upsilon(p,q) = 2$. Then $1/a_1 + 1/a_2$ is the unique best two-term underapproximation of $p/q$. 
\end{thm}

\begin{proof}
Note that $\Upsilon(p,q) = 2$ implies $p\geqslant 2$. If $q$ is odd and $p = 2$, then we have a contradiction to $p\mid (q+2)$. If $q$ is even and $p = 2$, then $p/q$ has the form $1/n$ for some $n\in \mathbb{N}$. We are done by Theorem \ref{Nathm}. For the rest of the proof, suppose that $p\geqslant 3$.
We can easily compute that $a_1 = (q+2)/p$ and 
$$a_2 \ =\ G\left(\frac{p}{q} - \frac{1}{a_1}\right)\ =\ G\left(\frac{p}{q} - \frac{p}{q+2}\right)\ =\ G\left(\frac{2p}{q(q+2)}\right)\ =\ \frac{q(q+2)}{2p} + \frac{v}{2},$$
where 
$$v \ =\ \begin{cases} 1 &\mbox{ if }2\nmid q,\\ 2 &\mbox{ if }2\mid q.\end{cases}$$
Suppose that there are integers $2\leqslant x_1\leqslant x_2$ such that $(x_1, x_2)\neq (a_1, a_2)$ and \eqref{ep8} holds with $\theta = p/q$. By Lemma \ref{lp3}, 
\begin{equation}\label{ep10}G\left(\frac{p}{q} - \frac{1}{x_1}\right)\ \leqslant\ \left(\frac{p}{q+2} + \frac{2p}{q(q+2)+vp}-\frac{1}{x_1}\right)^{-1}.\end{equation}
By Lemma \ref{lp2}, $a_1+1\leqslant x_1\leqslant 2a_1-1$. Hence, we can write $x_1 = a_1 + s$ for some $1\leqslant s\leqslant a_1-1$.
Inequality \eqref{ep10} becomes
$$G\left(\frac{p}{q} - \frac{1}{a_1+s}\right)\ \leqslant\ \left(\frac{p}{q+2} + \frac{2p}{q(q+2)+vp}-\frac{1}{a_1+s}\right)^{-1}.$$
Replacing $p$ by $(q+2)/a_1$, we have
\begin{equation}\label{ep11}\left\lfloor \frac{qa_1(a_1+s)}{2a_1+2s+qs}\right\rfloor\ \leqslant\ \frac{a_1(a_1+s)(qa_1+v)}{sqa_1+sv+2a_1(a_1+s)}-1.\end{equation}

Let us check that we can apply Lemma \ref{lp1}. 
Since $q > p\geqslant 3$, we know that $q\geqslant 4$. Furthermore, $a_1\geqslant 2$ by the definition of $G$. Finally, $(q+2)/a_1 = p\geqslant 3$. Therefore, Lemma \ref{lp1} states that \eqref{ep11} cannot hold. This contradiction completes our proof.
\end{proof}

\subsection{When $\Upsilon(p,q) = 3$}
We need an analog of Lemma \ref{lp1} for the case $\Upsilon(p,q) = 3$. The proof of the next lemma is similar, albeit more technically involved, to the proof of Lemma \ref{lp1}. We include the proof in the appendix for interested readers. 

\begin{lem}\label{lp11}
Let $q \geqslant 5$, $u\geqslant 2$, $v\in \{1,2,3\}$, and $1\leqslant s\leqslant u-1$. If $(q+3)/u\in \mathbb{N}_{\geqslant 4}$ and $(q, u)\notin \{(17,2), (61,8)\}$, then 
\begin{equation}\label{ep30}\left\lfloor \frac{qu(u+s)}{s(q+3)+3u}\right\rfloor \ >\ \frac{(qu+v)u(u+s)}{squ+vs+3u(u+s)}-1.\end{equation}
\end{lem}

The next easy lemma will also be used in due course. 

\begin{lem}\label{lp12}
For $s\geqslant 1$ and $s\neq 155$, it holds that
\begin{equation}\label{ep63}\left\lfloor \frac{61(8+s)}{8s+3}\right\rfloor\ >\ \frac{3912(8+s)}{513s+192}-1.\end{equation}
\end{lem}
\begin{proof}
The left side of \eqref{ep63} can be written as
$$\left\lfloor 8-\frac{3s-464}{8s+3}\right\rfloor,$$
while the right side is equal to 
$$7-\frac{192s-29760}{513s+192}.$$
Since $(3s-464)/(8s+3) < 1$, the left side is at least $7$. For $s \geqslant 156$, the right side is strictly smaller than $7$; hence, the inequality holds for all $s\geqslant 156$. A quick computer check proves the inequality for $s\leqslant 154$. 
\end{proof}

\begin{thm}\label{lc}
Let $p, q\in \mathbb{N}$ with $p < q$, $\Upsilon(p,q) = 3$, and $p/q\neq 10/17$. Then $1/a_1 + 1/a_2$ is the unique best two-term underapproximation of $p/q$. 
\end{thm}

\begin{proof}
We shall use the same argument as in the proof of Theorem \ref{mk01}, where we assume that $1/a_1+1/a_2$ is not the unique best two-term underapproximation to establish an upper bound for $G(p/q-1/(a_1+s))$ for $1\leqslant s\leqslant a_1-1$. We then use Lemma \ref{lp11} to obtain a contradiction. Here the role of Lemma \ref{lp11} in the proof of Theorem \ref{lc} is the same as the role of Lemma \ref{lp1} in the proof of Theorem \ref{mk01}.

Since $\Upsilon(p,q) = 3$, we know that $p\geqslant 3$. If $p = 3$, then $\Upsilon(p,q) = 3$ implies that $3\mid q$. Then $p/q$ has the form $1/n$, and we are done using Theorem \ref{Nathm}. For the rest of the proof, assume that $p\geqslant 4$. It is easy to verify that $a_1 = (q+3)/p$ and 
$$a_2 \ =\ \left\lfloor \left(\frac{p}{q}-\frac{p}{q+3}\right)^{-1}\right\rfloor + 1 \ =\ \left\lfloor \frac{q(q+3)}{3p}\right\rfloor+1\ =\ \frac{q(q+3)}{3p}+\frac{v}{3},$$
for some $v\in \{1,2,3\}$. Suppose that there are integers $2\leqslant x_1\leqslant x_2$ such that $(x_1, x_2)\neq (a_1, a_2)$ and \eqref{ep8} holds. By Lemma \ref{lp3}, 
\begin{equation}\label{ep61}G\left(\frac{p}{q} - \frac{1}{x_1}\right)\ \leqslant\ \left(\frac{p}{q+3} + \frac{3p}{q(q+3)+vp}-\frac{1}{x_1}\right)^{-1}.\end{equation}
By Lemma \ref{lp2}, $(q+3)/p+1\leqslant x_1\leqslant 2(q+3)/p-1$. Hence, we can write 
$$x_1 \ =\ (q+3)/p+ s,$$
for some $1\leqslant s\leqslant (q+3)/p-1$. Replacing $x_1$ by $(q+3)/p+s$ in \eqref{ep61} gives
$$\left\lfloor\frac{q(ps+q+3)}{3p+p^2s}\right\rfloor \ \leqslant\ \frac{\frac{q+3}{p}\left(\frac{q+3}{p}+s\right)\left(q\frac{q+3}{p}+v\right)}{qs\frac{q+3}{p}+sv+3\frac{q+3}{p}\left(\frac{q+3}{p}+s\right)}-1.$$
Replacing $(q+3)/p$ by $a_1$, we obtain
\begin{equation}\label{ep65}\left\lfloor\frac{qa_1(a_1+s)}{s(q+3)+3a_1}\right\rfloor\ \leqslant\ \frac{a_1(a_1+s)(qa_1+v)}{sqa_1+sv+3a_1(a_1+s)}-1.\end{equation}
If $(p,q) = (8,61)$, then $a_1 = 8, v= 1$, and \eqref{ep65} cannot hold due to Lemma \ref{lp12} and the fact that $s\leqslant 7$. Supposing that $(p, q)\neq (8,61)$, we check that Lemma \ref{lp11} can be applied. We have $q > p\geqslant 4$, so $q\geqslant 5$. Clearly, $a_1\geqslant 2$ and $1\leqslant s\leqslant a_1-1$. Furthermore, $(q+3)/u = p\geqslant 4$. Finally, we need to check that $(q,u)\notin \{(17,2), (61,8)\}$. If $(q,u) = (17,2)$, then $p = (q+3)/u = 10$; however, $(p,q)\neq (10,17)$ by hypothesis. If $(q, u) = (61,8)$, then $p = (q+3)/u = 8$; however, $(p,q)\neq (8, 61)$ by our supposition. Therefore, Lemma \ref{lp11} states that \eqref{ep65} does not hold. Hence, there are no integers $2\leqslant x_1\leqslant x_2$ such that $(x_1, x_2)\neq (a_1, a_2)$ and \eqref{ep8} holds. This completes our proof. 
\end{proof}

\section{The best two-term underapproximation when $\Upsilon(p,q) \geqslant 4$}\label{>3}
In this section, we prove the third statement of Theorem \ref{mk0}. In particular, for every $k\geqslant 4$, we show the existence of an integer $q > k+1$ such that 
\begin{enumerate}
\item[i)] $\Upsilon(k+1,q) = k$,
\item[ii)] $k$ divides $q$, and 
\item[iii)] $\mathcal{G}$ does not give the best two-term underapproximation for $(k+1)/q$; specifically,
\begin{equation}\label{s3}\frac{1}{a_1} + \frac{1}{a_2}\ <\ \frac{1}{a_1+1} + \frac{1}{G((k+1)/q - 1/(a_1+1))}\ <\ \frac{k+1}{q},\end{equation}
where $(a_n)_{n=1}^\infty = \mathcal{G}((k+1)/q)$.
\end{enumerate}

Fix $k\geqslant 4$. The set of $q$ that is greater than $k+1$ and satisfies Condition i) is 
$$\{(k+1)u-k\,:\, u\geqslant 2\}.$$
Hence, the set of $q$ that is greater than $k+1$ and satisfies both Condition i) and Condition ii) is
$$\{(k+1)kv - k\,:\, v\geqslant 1\}.$$
The problem comes to proving the existence of $v\geqslant 1$ such that $q = (k+1)kv-k$ satisfies Condition iii); i.e.,
$$\frac{1}{a_1} + \frac{1}{a_2}\ <\ \frac{1}{a_1+1} + \frac{1}{G((k+1)/((k+1)kv-k) - 1/(a_1+1))}\ <\ \frac{k+1}{(k+1)kv-k}.$$
We easily compute that
$$\begin{cases}a_1 &\ =\ kv,\\ a_2&\ =\ kv((k+1)v-1)+1,\\ G\left(\frac{k+1}{(k+1)kv-k} - \frac{1}{a_1+1}\right)&\ =\ \left\lfloor\frac{k((k+1)v-1)(kv+1)}{2k+1}\right\rfloor + 1.\end{cases}$$
Hence, \eqref{s3} is equivalent to 
$$\frac{1}{kv} + \frac{1}{kv((k+1)v-1)+1} - \frac{1}{kv+1} \ <\ \left(\left\lfloor\frac{k((k+1)v-1)(kv+1)}{2k+1}\right\rfloor + 1\right)^{-1},$$
which can be rewritten as
$$\frac{k(kv+1)((k+1)v-1)+k+1/v}{2k+1+1/(kv^2)}\ >\ \left\lfloor\frac{k(kv+1)((k+1)v-1)}{2k+1}\right\rfloor + 1.$$
Therefore, it suffices to prove the following theorem.
\begin{thm}\label{fthm}
For every $k\geqslant 4$, there exists $v\geqslant 1$ such that 
\begin{equation}\label{s5}\frac{k(kv+1)((k+1)v-1)+k+1/v}{2k+1+1/(kv^2)}\ >\ \left\lfloor\frac{k(kv+1)((k+1)v-1)}{2k+1}\right\rfloor + 1.\end{equation}
\end{thm}

We establish Theorem \ref{fthm} through the next four subsections, each of which corresponds to a case of $k\mod 4$.

\subsection{When $k\equiv 0 \mod 4$}
Write $k = 4j$ for some $j\geqslant 1$. In that case, $v = 1$ is a solution to \eqref{s5}. Indeed, plugging $k = 4j$ and $v = 1$ into \eqref{s5} gives
$$\frac{16j^2(4j+1)+4j+1}{8j+1+1/(4j)}\ >\ \left\lfloor \frac{16j^2(4j+1)}{8j+1}\right\rfloor+1,$$
which is the same as 
$$\frac{64j^3(4j+1)+16j^2+4j}{32j^2+4j+1}\ >\ 8j^2+j.$$
It is easy to check that the last inequality holds. 

\subsection{When $k\equiv 1\mod 4$}\label{sss1} Write $k = 4j+1$ for some $j\geqslant 1$. Choose $s\geqslant 1$ such that 
$$\frac{s(s+1)}{2}\ \leqslant\ j\ \leqslant\ \frac{(s+1)(s+2)}{2}-1.$$
Choose $v = 2s+5$. Our goal is to show that $(k,v) = (4j+1, 2s+5)$ is a solution to \eqref{s5}. 
The quantity inside the floor function in \eqref{s5} is
$$x\ :=\ \frac{(4j+1)((4j+1)(2s+5)+1)((4j+2)(2s+5)-1)}{2(4j+1)+1}.$$

\begin{claim}\label{cls1} For $j\geqslant 1$, 
$$x-\left\lfloor x\right\rfloor\ =\ \frac{5j+2+3s+(s-2)(s-1)/2}{8j+3}.$$
\end{claim}

\begin{proof}
Due to $j\geqslant s(s+1)/2$, we have
$$0\ <\ \frac{5j+2+3s+(s-2)(s-1)/2}{8j+3}\ <\ 1.$$
It remains to show that 
$$x - \frac{5j+2+3s+(s-2)(s-1)/2}{8j+3}\ \in\ \mathbb{N}.$$
This is indeed true, as the difference is
$$32j^2s^2+160j^2s+200j^2+20js^2+100js+125j+\frac{1}{2}(5s^2+27s)+17.$$
Note that $5s^2+27s$ is even for all $s$. 
\end{proof}

Due to the proof of Claim \ref{cls1}, the right side of \eqref{s5} is equal to
\begin{align*}
R&\ :=\ \lfloor x\rfloor + 1\\
&\ =\ 32j^2s^2+160j^2s+200j^2+20js^2+100js+125j+\frac{1}{2}(5s^2+27s)+18,
\end{align*}

Then \eqref{s5} is equivalent to 
\begin{align*}
&(4j+1)((4j+1)(2s+5)+1)((4j+2)(2s+5)-1)+(4j+1)+\frac{1}{2s+5}\\
&\ >\ \left(8j+3+\frac{1}{(4j+1)(2s+5)^2}\right)R.
\end{align*}
\textsc{Mathematica} simplifies the above inequality to a nice quadratic inequality in $j$:
$$(8s^2+40s+50)j^2 - (4s^4+32s^3+85s^2+79s+10)j - (s^4+8s^3+22s^2+23s+6) \ <\ 0.$$
It suffices to verify that $j$ lies strictly between the two roots of the equation
$$(8s^2+40s+50)x^2 - (4s^4+32s^3+85s^2+79s+10)x - (s^4+8s^3+22s^2+23s+6) \ =\ 0.$$
The two roots $x_1 < x_2$ are
$$\frac{4s^4+32s^3+85s^2+79s+10\pm\sqrt{\Delta}}{16s^2+80s+100},$$
where $\Delta = 16s^8+256s^7+1736s^6+6488s^5+14545s^4+19926s^3+16213s^2+7140s+1300$. Observe that $x_1 < 0$ for all $s\geqslant 1$, so $j > x_1$. On the other hand, 
$$j \ \leqslant\ \frac{(s+1)(s+2)}{2}-1\ <\ x_2.$$
This completes our proof of Theorem \ref{fthm} in the case $k\equiv 1\mod 4$. 

\begin{figure}[H]
\centering
\includegraphics[scale=0.75]{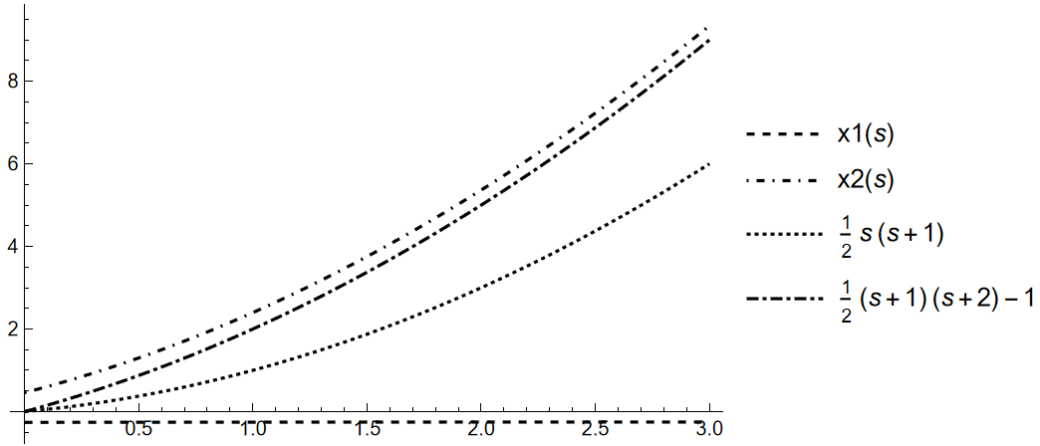}
\caption{The two bounds for $\ell$ are strictly between $x_1$ and $x_2$ for any value of $s$.}
\end{figure}

\subsection{When $k\equiv 2\mod 4$}\label{sss2}
Write $k = 4j+2$ for some $j\geqslant 1$. A brute-force calculation confirms \eqref{s5} for all $1\leqslant j\leqslant 11$. Examples of $(k,v)$ are shown in the following table. 
\begin{table}[H]
\centering
\begin{tabular}{ |c|c|c|c|c|c|c|c|c|c|c|c| } 
 \hline
 $k$ & $6$ & $10$ & $14$ & $18$ & $22$ & $26$ & $30$ & $34$ & $38$ & $42$ & $46$  \\ 
 $v$ & $8$ & $11$ & $12$ & $12$ & $12$ & $15$ & $16$ & $16$ & $16$ & $16$ & $16$  \\ 
 \hline
\end{tabular}
\caption{Values of $(k,v)$ ($6\leqslant k\leqslant 46$) that satisfies \eqref{s5}.}
\end{table}
We assume that $j\geqslant 12$. Choose $s\geqslant 0$ such that 
$$12+s(s+7)\ \leqslant \ j\ \leqslant\ 11+(s+1)(s+8).$$
Let $v = 4s+20$. 
We shall show that $(k,v) = (4j+2, 4s+20)$ is a solution of \eqref{s5}. The quantity inside the floor function in \eqref{s5} is
$$x\ :=\ \frac{(4j+2)((4j+2)(4s+20)+1)((4j+3)(4s+20)-1)}{8j+5}.$$
\begin{claim}\label{cls2}
For $j\geqslant 12$,
$$x-\lfloor x\rfloor\ =\ \frac{2s^2+18s+4j+43}{8j+5}.$$
\end{claim}
\begin{proof}
It follows from $j\geqslant 12+s(s+7)$ that
$$0\ <\ \frac{2s^2+18s+4j+43}{8j+5} \ <\ 1.$$
Furthermore, \textsc{Mathematica} computes
$x-(2s^2+18s+4j+43)/(8j+5) = $ 
$$959 + 3600j + 3200 j^2 + 382 s + 1440 j s + 1280 j^2 s + 38 s^2 + 
 144 j s^2 + 128 j^2 s^2\in \mathbb{N}.$$
 This completes our proof. 
\end{proof}

Due to the proof of Claim \ref{cls2}, the right side of \eqref{s5} is
$$R\ :=\ 960 + 3600j + 3200 j^2 + 382 s + 1440 j s + 1280 j^2 s + 38 s^2 + 
 144 j s^2 + 128 j^2 s^2.$$
Plugging $(k, v) = (4j+2, 4s+20)$ into \eqref{s5}, we need to confirm that
\begin{align*}&(4j+2)((4j+2)(4s+20)+1)((4j+3)(4s+20)-1)+4j+2+\frac{1}{4s+20}\\
&\ >\ \left(8j+5+\frac{1}{(4j+2)(4s+20)^2}\right)R.
\end{align*}
\textsc{Mathematica} simplifies the above inequality to
$$(64s+320)j^2-(64s^3+896s^2+4088s+6048)j - (32s^3+448s^2+2061s+3108)\ <\ 0.$$
It suffices to show that $j$ lies strictly between the two roots of the equation
$$(64s+320)x^2-(64s^3+896s^2+4088s+6048)x - (32s^3+448s^2+2061s+3108)\ =\ 0.$$
The two roots $x_1 < x_2$ are
$$
\frac{32s^3+448s^2+2044s+3024\pm\sqrt{\Delta}}{64s+320},
$$
where $$\Delta \ =\ 16 (64s^6+1792s^5+20848s^4+128992s^3+447669s^2+826284s+633696).$$
Since $x_1 < 0$, we have $x_1 < j$. By our choice of $s$,
$$j\ \leqslant\ 11+(s+1)(s+8)\ <\ x_2.$$
We are done with the case $k\equiv 2\mod 4$.

\subsection{When $k\equiv 3\mod 4$}
Write $k = 4j+3$ for some $j\geqslant 1$. A brute-force calculation confirms \eqref{s5} for all $1\leqslant j\leqslant 5$. Examples of $(k,v)$ are shown in the following table. 
\begin{table}[H]
\centering
\begin{tabular}{ |c|c|c|c|c|c|c|c|c|c|c|c| } 
 \hline
 $k$ & $7$ & $11$ & $15$ & $19$ & $23$  \\ 
 $v$ & $8$ & $13$ & $12$ & $12$ & $12$  \\ 
 \hline
\end{tabular}
\caption{Values of $(k,v)$ ($7\leqslant k\leqslant 23$) that satisfies \eqref{s5}.}
\end{table}
We assume that $j\geqslant 6$. Choose $s\geqslant 2$ such that 
$$s(s+1)\ \leqslant\ j\ \leqslant\ (s+1)(s+2)-1.$$
Let $v = 4s+8$. The proof of the following theorem is similar to what we have done in Subsections \ref{sss1} and \ref{sss2}, so we move it to the appendix for interested readers.

\begin{thm}\label{fthm3}
For $j\geqslant 6$, $(k, v) = (4j+3, 4s+8)$ is a solution of \eqref{s5}.
\end{thm}

\section{Stepwise underapproximation by $\mathcal{G}$}\label{compare}
This section analyzes the performance of $\mathcal{G}$ at each step by proving Theorem \ref{mk} and its complementary result, Proposition \ref{mk2}.

\begin{proof}[Proof of Theorem \ref{mk}]
We prove that i) implies ii) by contrapositive. Suppose that 
$$\frac{N}{Na_m} \ =\ \frac{1}{a_m} \ <\ \frac{N}{b_m}\ <\ e_{m-1},$$
which, combined with \eqref{e13}, gives
$$\frac{N}{Na_m - 1}\ \leqslant\ \frac{N}{b_m}\ <\ e_{m-1}\ =\ \frac{1}{a_m} + \sum_{n=m+1}^\infty\frac{1}{a_n}\ \leqslant\ \frac{1}{a_m} + \frac{1}{a_{m+1}-1}.$$
Rearranging $N/(Na_m-1)\ <\ 1/a_m + 1/(a_{m+1}-1)$, we obtain
$$a_{m+1} \ <\ N a_m^2 - a_m + 1.$$

Next, we show that ii) implies iii). By ii), $N a_m$ is the smallest positive integer such that $N/(Na_m) < e_{m-1}$. Hence,
$$e_{m-1}\ \in\ \left(\frac{1}{a_m}, \frac{N}{Na_m-1}\right].$$
We have
$$\Phi(e_{m-1}) \ =\ \frac{1}{G(e_{m-1})-1/e_{m-1}}\ =\ \frac{1}{a_m - 1/e_{m-1}}\ \geqslant\ \frac{1}{a_m - (Na_m-1)/N}\ =\ N.$$

Next, we assume iii) to derive iv). By iii), 
$$N \ \leqslant\ \frac{1}{G(e_{m-1})- 1/e_{m-1}}\ =\ \frac{e_{m-1}}{e_{m-1}G(e_{m-1})-1}\ =\ \frac{e_{m-1}}{e_{m-1}a_m-1}.$$
Rearranging $N\leqslant e_{m-1}/(e_{m-1}a_m-1)$ gives iv). 

Finally, we verify that iv) implies i). We have 
$$\frac{N}{N a_m - 1}\ \geqslant\ e_{m-1} \ =\ \sum_{n=m}^\infty \frac{1}{a_n}\ >\ \frac{1}{a_m} + \frac{1}{a_{m+1}}.$$
Subtracting $1/a_m$ from both sides to get
$$\frac{1}{Na_m^2 - a_m}\ >\ \frac{1}{a_{m+1}},$$
which gives i). 
\end{proof}

\begin{cor}
Let the sequence $(c_n)_{n=1}^\infty$ be such that $c_1\geqslant 2$ and $c_{n+1}\geqslant c_n^3 - c_n + 1$. Fix $N\in \mathbb{N}$ and let $m_N$ be the smallest $m$ such that $c_m\geqslant N$. Then $\theta := \sum_{n=1}^\infty 1/c_n$ has the property that whenever $m\geqslant m_N$, we have $1/a_m = N/b_m$. (Since $(c_n)_{n=1}^\infty$ is strictly increasing, it follows that $m_N\leqslant N$.)
\end{cor}
\begin{proof}
Use Theorem \ref{mk}.
\end{proof}

Before proving Proposition \ref{mk2} and giving examples to show its optimality, we make a remark about the following theorem by Nathanson.
\begin{thm} \cite[Theorem 1]{Na23}\label{mk6}
Let $\theta = p/q\in (0,1]$, where $p$ and $q$ are positive integers such that $p$ divides $q+1$, and let $(a_n)_{n=1}^\infty = \mathcal{G}(\theta)$. Then 
\begin{equation}\label{e-32}\tag{6}a_1 \ = \ \frac{q+1}{p}\mbox{ and }a_{n+1} \ =\ q\prod_{i=1}^{n}a_i+1, \quad n\geqslant 1.\end{equation}
\end{thm}
\begin{rek}\normalfont\label{onr}
When $p = 1$, note that \eqref{e-32} gives 
\begin{equation*}a_1 \ = \ q+1\mbox{ and }a_{n+1} \ =\ a_{n}^2 - a_n + 1, \quad n\geqslant 1.\end{equation*}
Indeed, setting $M = q\prod_{i=1}^{n-1}a_i$, we use \eqref{e-32} to have $a_{n+1} = a_{n}M+1$ and $a_{n} = M+1$. Hence,
\begin{align*}
a_{n}^2 - a_{n} + 1&\ =\ (M+1)^2 - (M+1)+1 \ =\ M^2+M+1\\
&\ =\ M(M+1)+1 \ =\ a_nM+1 \ =\ a_{n+1}.
\end{align*}
\end{rek}
\begin{proof}[Proof of Proposition \ref{mk2}]
We prove by induction on $\ell$. 

Base case: $\ell = 0$. Then $p$ divides $q$, so $\gcd(p, q) = 1$ implies that $p  = 1$. Hence, $(\theta)^{-1} = q\in \mathbb{N}$, as claimed. 

Inductive hypothesis: suppose that the theorem is true for all $\ell \leqslant k$, for some $k\geqslant 0$. 

We prove the theorem for $\ell = k+1$.
Pick $\theta = p/q\in (0,1]$, where $p, q\in \mathbb{N}$ and $\gcd (p, q) = 1$. Suppose that $k+1$ is the smallest nonnegative integer such that $p$ divides $q+k+1$. Let $(a_n) = \mathcal{G}(p/q)$. Clearly, $a_1 = (q+k+1)/p$. We have
\begin{align*}
{\theta} - \frac{1}{a_1} \ =\ \frac{p}{q} - \frac{p}{q+k+1}\ =\ \frac{k+1}{q\cdot \frac{q+k+1}{p}}\ =\ \frac{u}{v},
\end{align*}
for some $u, v\in \mathbb{N}$, $\gcd(u, v) = 1$, and $u\leqslant k+1$. Hence, if $j$ is the smallest nonnegative such that $u$ divides $v+j$, then $j\leqslant k$. Let $(b_n) = \mathcal{G}(u/v)$. By the inductive hypothesis, 
$$\left(\frac{u}{v} - \sum_{n=1}^{j}\frac{1}{b_n}\right)^{-1}\ \in\ \mathbb{N}.$$
Replacing $u/v$ by $\theta - 1/a_1$ and observing that $b_n = a_{n+1}$ give
$$\left(\theta - \sum_{n=1}^{j+1}\frac{1}{a_n}\right)^{-1}\ \in\ \mathbb{N}.$$
As $j+1\leqslant k+1$, we are done. Finally, \eqref{e-30} follows from Remark \ref{onr}.
\end{proof}

We evaluate the optimality of the integer $\ell$ in Proposition \ref{mk2}. For $\theta = p/q\in (0,1]$ $(p,q\in \mathbb{N})$, let $\Delta(\theta)$ be the smallest nonnegative number $m$ such that
$$\left(\frac{p}{q} - \sum_{n=1}^m\frac{1}{a_n}\right)^{-1}\ \in\ \mathbb{N},$$
where $(a_n)_n = \mathcal{G}(\theta)$.
Let $u/v$ be $p/q$ in its lowest terms and let $\ell$ be the smallest nonnegative integer such that $u$ divides $v + \ell$. 
According to Proposition \ref{mk2}, $\Delta(\theta) \ \leqslant \ \ell$. We show that in general, we cannot do better. 

\begin{prop}\label{mk3}
For every $\ell\in \mathbb{Z}_{\geqslant 0}$, there exists a fraction $p/q\in (0,1]$ that satisfies three conditions
\begin{enumerate}
\item \label{ccce}$\gcd(p,q) = 1$,
\item  \label{cccz} $\ell$ is the smallest nonnegative integer such that $p$ divides $q + \ell$, 
\item \label{ccc} $\Delta(p/q) = \ell$.
\end{enumerate}
\end{prop}

\begin{proof}
Let $\ell\geqslant 0$ be chosen. The case $\ell = 0$ is trivial as we can let $p = 1$ and $q = n$ for some $n\in \mathbb{N}$. Suppose that $\ell\geqslant 1$. We shall prove by induction that fractions of the form
$$\frac{\ell+1}{(\ell+1)!k+1}, \quad k\in \mathbb{N}$$
satisfy Condition \eqref{ccc} in Proposition \eqref{mk3}. (Conditions \eqref{ccce} and \eqref{cccz} are obviously satisfied.)

Base case: $\ell = 1$. Pick $k\in \mathbb{N}$ and consider the fraction $2/(2k+1)$. Clearly, $\gcd(2, 2k+1) = 1$, and $\ell = 1$ is the smallest nonnegative integer such that $2$ divides $(2k+1)+\ell$. Furthermore,
$\Delta(2/(2k+1))\leqslant 1$ by Proposition \ref{mk2}, but $(2/(2k+1))^{-1}\notin \mathbb{N}$. Hence, $\Delta(\theta) = 1 = \ell$. 

Inductive hypothesis: suppose that there exists $m\geqslant 1$ such that the claim holds for all $\ell\leqslant m$.

We show the claim is true for $\ell = m+1$. Pick $k\in \mathbb{N}$ and consider the fraction $$\frac{m+2}{(m+2)!k+1}.$$
Clearly, $\gcd(m+2, (m+2)!k+1) = 1$, and $\ell = m+1$ is the smallest nonnegative integer so that $m+2$ divides $(m+2)!k+1+\ell$. Clearly, the first term of $\mathcal{G}((m+2)/((m+2)!k+1))$ is $(m+1)!k+1$. We have
\begin{align*}\frac{m+2}{(m+2)!k+1} - \frac{1}{(m+1)!k+1}&\ =\ \frac{m+1}{((m+1)!k+1)((m+2)!k+1)}\\
&\ =\ \frac{m+1}{(m+1)!(k((m+2)!k+1)+(m+2)k)+1}.
\end{align*}
By the inductive hypothesis,
$$\Delta\left(\frac{m+1}{(m+1)!(k((m+2)!k+1)+(m+2)k)+1}\right)\ =\ m.$$
Therefore, 
$$\Delta\left(\frac{m+2}{(m+2)!k+1}\right)\ =\ m+1,$$
as desired. 
\end{proof}

\begin{exa}\normalfont
Though our bound for $\Delta$ cannot be improved in general, there are rational numbers $p/q$ with $\Delta(p/q) < \ell$. Consider $p/q = 7/54$. Here $\ell = 2$, but $(7/54-1/8)^{-1} = 216 \in \mathbb{N}$.
\end{exa}

\section{Computing $\mathcal{G}(\theta)$ for a family of $\theta\in (0,1]$}\label{gen}
While Proposition \ref{mk2} and Remark \ref{onr} guarantee that for any rational $\theta\in (0,1]$, the sequence $(a_n)_{n=1}^\infty = \mathcal{G}(\theta)$ eventually satisfies the recurrence relation $a_{n+1} = a_n^2 - a_n + 1$, we do not know if there is always a nice relation among the starting terms. Interestingly, for the special case when $\theta = p/q$, where $p$ divides $q+1$, 
Theorem \ref{mk6} gives an explicit formula to compute each term. In this section, we prove Theorem \ref{mk4}, a generalization of Theorem \ref{mk6}.  

\begin{proof}[Proof of Theorem \ref{mk4}]
Noting that $p\geqslant \Upsilon(p,q)$, we have
$$a_1 \ =\ \left\lfloor \frac{q}{p}\right\rfloor+1\ =\ \left\lfloor \frac{q+\Upsilon(p,q)}{p}-\frac{\Upsilon(p,q)}{p}\right\rfloor+1\ =\  \frac{q+\Upsilon(p,q)}{p}-1+1\ =\ \frac{q+\Upsilon(p,q)}{p}.$$
It suffices to prove the following equality
\begin{equation}\label{ef3}\frac{p}{q} - \sum_{i=1}^{n-1}\frac{1}{a_i} \ =\ \frac{1}{\frac{q}{\Upsilon(p,q)}\prod_{i=1}^{n-1}a_i}, \quad n\geqslant 2.\end{equation}
We proceed by induction. 

Base case: when $n = 2$, both sides are equal to $p\Upsilon(p,q)/(q(q+\Upsilon(p,q)))$. 

Inductive hypothesis: suppose that \eqref{ef3} holds for $n= k$ for some $k\geqslant 2$; i.e., 
\begin{equation}\label{es1}\frac{p}{q} - \sum_{i=1}^{k-1}\frac{1}{a_i} \ =\ \frac{1}{\frac{q}{\Upsilon(p,q)}\prod_{i=1}^{k-1}a_i}.\end{equation}
It follows that
$$a_{k} \ =\ \frac{q}{\Upsilon(p,q)}\prod_{i=1}^{k-1}a_i+1.$$
Hence, 
\begin{align*}\frac{p}{q} - \sum_{i=1}^{k}\frac{1}{a_i}&\ =\ \frac{1}{\frac{q}{\Upsilon(p,q)}\prod_{i=1}^{k-1}a_i}-\frac{1}{\frac{q}{\Upsilon(p,q)}\prod_{i=1}^{k-1}a_i+1}\\
&\ =\ \frac{1}{\frac{q}{\Upsilon(p,q)}\prod_{i=1}^{k-1}a_i\left(\frac{q}{\Upsilon(p,q)}\prod_{i=1}^{k-1}a_i+1\right)}\ =\ \frac{1}{\frac{q}{\Upsilon(p,q)}\prod_{i=1}^k a_i}.\end{align*}
This proves \eqref{ef3}, which in turns implies \eqref{ef1}.
\end{proof}

\begin{proof}[Proof of Proposition \ref{pl1}]
The proof is similar to that of Theorem \ref{mk4}. The only difference is to verify the base case of the following equality
\begin{equation}\label{s2}\frac{p}{q}-\sum_{i=1}^{k-1}\frac{1}{a_i}\ =\ \frac{1}{q\prod_{i=1}^{k-1}a_i}, \quad k\geqslant 3.\end{equation}
Let us check that the equality indeed holds for $k=3$ and let interested readers finish the proof.
For $k=3$, \eqref{s2} is equivalent to
$$\frac{2p\left(\lfloor qa_1/2\rfloor+1\right)-q(q+2)}{q(q+2)\left(\lfloor qa_1/2\rfloor+1\right)}\ =\ \frac{p}{q(q+2)(\lfloor qa_1/2\rfloor+1)}.$$
We, therefore, need to verify that
$$2p\left(\lfloor qa_1/2\rfloor+1\right)-q(q+2)\ =\ p.$$
Replacing $a_1$ by $(q+2)/p$ gives
$$2p\left\lfloor \frac{q(q+2)}{2p}\right\rfloor+p \ =\ q(q+2).$$
Equivalently,
$$\left\lfloor \frac{q(q+2)}{2p}\right\rfloor + \frac{1}{2}\ =\ \frac{q(q+2)}{2p},$$
which is true, because $q$ is odd. 
\end{proof}

\begin{exa}\normalfont
The sequence $(a_n)_n$ in Theorem \ref{mk4} and Proposition \ref{pl1} has the nice property that for sufficiently large $n$, $a_{n}-1$ is divisible by all $a_i$ for $i\leqslant n-1$. This is not true in general. Consider $\theta = 9/28$. The first $9$ terms of $\mathcal{G}(9/28)$ are
\begin{align*}&4, 15,  211, 44311, 1963420411, 3855019708367988511,\\ & 14861176951905611184725545411860008611,\\
&22085458039585055253184228908917552993753530939568130918727713764113414\\
&0711,\\
&48776745681827215201073590705821720129907215948452411133840819544687359\\
&97240857680148375403794922963882860912224317845389796211048815806159121\\
&3444811.
\end{align*}
None of these numbers is $1\mod 4$. Moreover, since $a_9\equiv 3\mod 4$,  Equation \eqref{e-30} gives that $a_n\equiv 3\mod 4$ for all $n\geqslant 10$. Hence, $a_n-1$ is not divisible by $a_1$.
\end{exa}

We end this paper by an application of Proposition \ref{pl1}.
For our proof of Theorem \ref{fthm2}, we borrow the argument from \cite{So05} and \cite[Theorem 5]{Na23}, whose key ingredient is the following corollary of Muirhead's inequality \cite{Na22}. For its proof, see \cite{AB15, Na23}. 
\begin{thm}\label{ki}
If $(x_i)_{i=1}^v$ and $(a_i)_{i=1}^v$ are increasing sequences of positive integers such that $(x_i)_{i=1}^v\neq (a_i)_{i=1}^v$ and 
$$\prod_{i=1}^{1+k} a_i\leqslant\ \prod_{i=1}^{1+k} x_i,$$
for all $0\leqslant k\leqslant v-1$, then
$$\sum_{i=1}^v\frac{1}{x_i} \ <\ \sum_{i=1}^v\frac{1}{a_i}.$$
\end{thm}

\begin{proof}[Proof of Theorem \ref{fthm2}]
Let $p<q\in \mathbb{N}$ such that $q$ is odd and $\Upsilon(p,q) = 2$. Let $(a_n)_{n=1}^\infty = \mathcal{G}(p/q)$. We shall prove, by induction on $m$, that for any increasing sequence of positive integers $x_1\leqslant \cdots\leqslant x_m$ with
$$\sum_{n=1}^{m}\frac{1}{a_n}\ \leqslant\ \sum_{n=1}^{m}\frac{1}{x_n} \ <\ \frac{p}{q},$$
we must have $a_n = x_n$ for all $1\leqslant n\leqslant m$.

Base case: when $m = 1$, the greedy algorithm guarantees that if $1/x_1 < p/q$, then $1/x_1\leqslant 1/a_1$. Hence, $1/x_1 = 1/a_1$ and so, $a_1 = x_1$. 

Inductive hypothesis: suppose that there exists $d\geqslant 1$ such that the claim is true whenever $m\leqslant d$.

Let us consider $m = d+1$. Assume that
\begin{equation}\label{ess3}\sum_{n=1}^{d+1}\frac{1}{a_n}\ \leqslant\ \sum_{n=1}^{d+1}\frac{1}{x_n} \ <\ \frac{p}{q},\end{equation}
which, combined with Proposition \ref{pl1}, gives
\begin{equation}\label{ess1}0\ <\ \frac{p}{q} - \sum_{n=1}^{d+1}\frac{1}{x_n}\ \leqslant\ \frac{p}{q} - \sum_{n=1}^{d+1}\frac{1}{a_n}\ =\ \frac{1}{q\prod_{n=1}^{d+1}a_n}.\end{equation}
Since $p/q - \sum_{n=1}^{d+1}1/x_n > 0$, we know that 
\begin{equation}\label{ess2}\frac{p}{q} - \sum_{n=1}^{d+1}\frac{1}{x_n}\ \geqslant\ \frac{1}{q\prod_{n=1}^{d+1} x_n}.\end{equation}
From \eqref{ess1} and \eqref{ess2}, we obtain
\begin{equation}\label{ess7}\prod_{n=1}^{d+1}a_n\ \leqslant\ \prod_{n=1}^{d+1}x_n.\end{equation}
Let $s$ be the largest positive integer at most $d+1$ such that 
$$\prod_{n=s}^{d+1}a_n\ \leqslant\ \prod_{n=s}^{d+1}x_n.$$
If $s = d+1$, then $1/a_{d+1} \geqslant 1/x_{d+1}$; hence,
$$\sum_{n=1}^d\frac{1}{a_n} \ \leqslant \ \sum_{n=1}^d \frac{1}{x_n}\ <\ \frac{p}{q}.$$
The inductive hypothesis guarantees $a_n = x_n$ for all $1\leqslant n\leqslant d$, which, along with \eqref{ess3}, gives
$1/a_{d+1}\leqslant 1/x_{d+1}$. By above, we must have $a_{d+1} = x_{d+1}$, and we are done. For the rest of the proof, we assume that $s \leqslant d$.
\begin{claim}\label{clll1}
For $j\in \{0, 1, \ldots, d-s\}$, it holds that
$$\prod_{n=s}^{s+j}a_n\ <\ \prod_{n=s}^{s+j} x_n.$$
\end{claim}
\begin{proof}
Suppose that for some $j_0\in \{0, 1, \ldots, d-s\}$, 
$$\prod_{n=s}^{s+j_0}a_n\ \geqslant\ \prod_{n=s}^{s+j_0} x_n.$$
Then
\begin{equation*}
\prod_{n=s+j_0+1}^{d+1} a_n\ \leqslant\ \frac{\prod_{n=s}^{d+1}x_n}{\prod_{n=s}^{s+j_0}a_n}\ \leqslant\ \frac{\prod_{n=s}^{d+1}x_n}{\prod_{n=s}^{s+j_0}x_n}\ =\ \prod_{n=s+j_0+1}^{d+1}x_n,
\end{equation*}
which contradicts the maximality of $s$.
\end{proof}

By Claim \ref{clll1}, we know that the two sequences $(a_n)_{n=s}^{d+1}$ and $(x_n)_{n=s}^{d+1}$ are distinct. 
Applying Theorem \ref{ki} to the two sequences, we have
\begin{equation}\label{ess4}\sum_{n=s}^{d+1}\frac{1}{x_n}\ <\ \sum_{n=s}^{d+1}\frac{1}{a_n}.\end{equation}
Therefore, $s$ must be at least $2$ for \eqref{ess3} to hold. Using \eqref{ess3} and \eqref{ess4}, we obtain
\begin{equation}\label{ess6}\sum_{n=1}^{s-1}\frac{1}{a_n}\ =\ \sum_{n=1}^{d+1}\frac{1}{a_n} - \sum_{n=s}^{d+1}\frac{1}{a_n}\ <\ \sum_{n=1}^{d+1}\frac{1}{x_n} - \sum_{n=s}^{d+1}\frac{1}{x_n}\ =\ \sum_{n=1}^{s-1}\frac{1}{x_n}.\end{equation}
The inductive hypothesis guarantees that $a_n = x_n$ for all $1\leqslant n\leqslant s-1$, which is absurd due to the strict inequality in \eqref{ess6}. This contradiction shows that $s = d+1$, and in this case, we already show that $a_n = x_n$ for all $1\leqslant n\leqslant d+1$. 
\end{proof}

\begin{rek}\normalfont
A nice problem for future investigation is whether it is true that when $\Upsilon(p,q)\in\{2,3\}$, $\mathcal{G}$ always gives the best $m$-term underapproximation for any $m\geqslant 3$. Theorems \ref{Nathm} and \ref{fthm2} give a positive answer for the case $\Upsilon(p,q) = 1$ and the case $\Upsilon(p,q) = 2$ and $q$ is odd, respectively. Similar to the proof of Theorem \ref{Nathm}, the proof of Theorem \ref{fthm2} employs the key Inequality \eqref{ess7}, resulting from the formula given in Proposition \ref{pl1}. This formula is different from \eqref{ef1}, where $q$ is divided by $\Upsilon(p,q)$. To prove the best $m$-term underapproximation for the case $\Upsilon(p,q) = 2$ and $q$ is even, we need to overcome this distinction. 
\end{rek}

\subsection*{Data availability}
No data is available for this paper.

\section{Appendix}
\subsection{Proof of Lemma \eqref{lp11}} Before proving Lemma \ref{lp11}, we need the following auxiliary lemma. 
\begin{lem}\label{lp50}
Let $q\geqslant 5$ and $u\geqslant 2$ with $(q+3)/u\in \mathbb{N}_{\geqslant 4}$. If $(q,u)\notin \{(17,2),(61,8)\}$, we have
\begin{equation}\label{ep60}\left\lfloor \frac{qu(u+1)}{q+3(u+1)}\right\rfloor\ >\ \frac{(qu+3)u(u+1)}{qu+3+3u(u+1)}-1.\end{equation}
\end{lem}

\begin{proof}
Fix $u\geqslant 2$ and let $q = uk-3$ for some $k\geqslant 4$. Substituting $q$ by $uk-3$ in \eqref{ep60}, the left side becomes
$$L\ :=\ \left\lfloor u(u+1)-\frac{3(u+1)^2}{k+3}\right\rfloor;$$
the right side becomes
$$R\ :=\ u(u+1)-\frac{3(u+1)^2}{k+3+3/u^2}-1.$$
Write 
\begin{equation}\label{ep70} 3(u+1)^2 = \ell(k+3)+j,\end{equation}
for $\ell\geqslant 0$ and $0\leqslant j\leqslant k+2$. If $j = 0$,
$$L - R\ =\ -\frac{3(u+1)^2}{k+3}+\frac{3(u+1)^2}{k+3+3/u^2}+1\ =\ \frac{3k+k^2u^2+6u(ku-3)}{(k+3)(3u^2+ku^2+3)} \ >\ 0.$$
If $j\geqslant 1$, 
$$L\ =\ u(u+1)-\ell-1\mbox{ and }R\ =\ u(u+1)-\ell-1-\frac{ju^2-3\ell}{(k+3)u^2+3}.$$
Therefore, $L-R > 0$ if and only if $ju^2 > 3\ell$, which we shall prove. Suppose otherwise that $ju^2\leqslant 3\ell$ and so, $u^2/\ell\leqslant 3/j$.
Then 
$$\frac{27}{4}u^2\ =\ 3\left(\frac{3u}{2}\right)^2\ \geqslant\ 3(u+1)^2\ =\ \ell(k+3)+j.$$
Dividing both sides by $\ell$ and using $u^2/\ell\leqslant 3/j$, we obtain
$$\frac{81}{4j}\ \geqslant\ \frac{27}{4}\frac{u^2}{\ell} \ \geqslant\ k+3+\frac{j}{\ell}\ >\ k+3.$$
Observe that $j\in \{1,2\}$ because $k\geqslant 4$. 

If $j = 1$, we have $4\leqslant k\leqslant 17$, and $k$ must satisfy $3(u+1)^2 = \ell(k+3)+1$. Hence, $u+1$ is a solution of the congruence equation 
$$3n^2\ \equiv\ 1 \mod (k+3),$$
which has a solution only for $k = 8, 10$ out of all integers between $4$ and $17$.

If $j =  2$, then $4\leqslant k\leqslant 7$. However, only  $k = 7$ makes the solution set of the congruence equation
$3n^2\ \equiv \ 2\mod (k+3)$ nonempty. We have narrowed the values of $k$ to $\{7, 8, 10\}$. We now examine each case separately.

Case 1: $j = 2$ and $k = 7$. By \eqref{ep70}, we have
$$\ell \ =\ \frac{3}{10}(u+1)^2-\frac{1}{5}.$$
Combining this with our assumption that $2u^2\leqslant 3\ell$, we obtain
$$11u^2-18u-3\ \leqslant \ 0,$$
which cannot hold for $u\geqslant 2$. 

Case 2: $j = 1$ and $k = 8$. The same argument as in Case 1 gives $u^2\leqslant 9u+3$; hence, $2\leqslant u\leqslant 9$. Only $u = 8$ gives an integer value for $\ell$, in which case $q = uk-3 = 61$. However, $(q, u)\neq (61,8)$ by our hypothesis. 

Case 3: $j = 1$ and $k = 10$. The same argument as the two above cases gives $2u^2\leqslant 9u+3$. Hence, $u \in \{2, 3, 4\}$. Only $u = 2$ gives an integer value for $\ell$ in \eqref{ep70}. It follows that $q = uk-3 = 17$. However, $(q, u)\neq (17,2)$ by our hypothesis. 

We have completed the proof. 
\end{proof} 

\begin{proof}[Proof of Lemma \ref{lp11}] Thanks to Lemma \ref{lp50}, we can exclude the case $(s, v) = (1,3)$ from consideration.
Fix $u\geqslant 2$. We prove \eqref{ep30} for all $q$ in the set 
$\left\{ku-3\,:\, k\geqslant 4\right\}$. Choose $k\geqslant 4$ and replace $q = ku-3$ in \eqref{ep30}. The left side of \eqref{ep30} becomes 
$$L\ :=\ \left\lfloor \frac{(ku-3)u(u+s)}{kus+3u}\right\rfloor\ =\ \left\lfloor u+\frac{ku^2-6u-3s}{ks+3}\right\rfloor.$$
The right side of \eqref{ep30} becomes
$$R\ := \ u + \frac{ku^4-6u^3-3u^2s+vu^2}{ku^2s+vs+3u^2}-1\ =\ u +\frac{ku^2-6u-3s+v}{ks+3+vs/u^2}-1.$$
We claim that $ku^2-6u-3s > 0$. Indeed, when $u = 2$, we have $s = 1$, and
$$ku^2-6u-3s \ =\ 4k-12-3 \ =\ 4k-15\ \geqslant\ 1;$$
when $u \geqslant 3$,
$$ku^2\ \geqslant\ 4u^2\ \geqslant\ 12u\ >\ 9u-3\ =\ 6u+3(u-1)\ \geqslant\ 6u+3s.$$
As in the proof of Lemma \ref{lp1}, we eliminate the floor function by writing 
\begin{equation}\label{ep50}ku^2-6u-3s \ =\ (ks+3)\ell+j,\end{equation}
for some $\ell\geqslant 0$ and $0\leqslant j\leqslant ks+2$. Then
$$L \ =\  u + \ell$$
and 
$$R\ =\ u+\ell + \frac{j+v-vs\ell/u^2}{ks+3+vs/u^2}-1.$$
Hence,
\begin{equation}\label{ep51}L- R\ =\ \frac{(ks-j)u^2+(3-v)u^2+vs+vs\ell}{(ks+3)u^2+vs}\ >\ 0.\end{equation}
Our goal is to show that the right side of \eqref{ep51} is positive. We proceed by case analysis. Similar to what we did in the proof of Lemma \ref{lp1}, we first classify values for $j$, then establish finite ranges for $u$, $s$, and $k$, and finally resort to a brute-force check over all possible triples $(u,s,k)$ to complete the proof. 

Case 1: If $0\leqslant j\leqslant ks$, then $L-R > 0$.

Case 2: If $j = ks+1$, then from \eqref{ep51}, 
$$L-R \ =\ \frac{(2-v)u^2+vs+vs\ell}{(ks+3)u^2+vs}.$$
If $v\leqslant 2$, $L-R>0$. Assume that $v = 3$. We need to show that $3s(\ell+1) > u^2$. Suppose, for a contradiction, that $3s(\ell+1) \leqslant u^2$. 
By \eqref{ep50},
\begin{align*}
    3ku^2 &\ =\ 3ks\ell+9\ell+3j+18u+9s\\
    &\ \leqslant\ k(u^2-3s) + 9(u^2/(3s)-1) + 3ks+3+18u+9(u-1)\\
    &\ =\ ku^2+3u^2/s+27u-15\ \leqslant\ (k+3)u^2+27u-15.
\end{align*}
Hence, 
\begin{equation*}5u^2\ \leqslant\ (2k-3)u^2 \ \leqslant\ 27u-15,\end{equation*}
which gives $2\leqslant u\leqslant 4$ and $4\leqslant k\leqslant 6$. For each value of $u$, $s$ is between $1$ and $u-1$. Feeding all these possible triples $(u, s, k)$ into \eqref{ep50}, we see that there is no triple $(u,s, k)$ that gives an integer value for $\ell$ and satisfies $3s(\ell+1)\leqslant u^2$. 

Case 3: $j = ks+2$. From \eqref{ep51}, 
$$L-R \ =\ \frac{(1-v)u^2+vs+vs\ell}{(ks+3)u^2+vs}.$$
If $v = 1$, $L-R > 0$. Assume that $v \geqslant 2$.

\begin{enumerate}
\item[i)] Case 3.1: $v = 2$. We shall show that $2s(\ell+1) > u^2$. Suppose, for a contradiction, that $2s(\ell+1) \leqslant u^2$. According to \eqref{ep50}, 
\begin{align*}
2ku^2&\ =\ 2ks\ell+6\ell+2j + 12u + 6s\\
&\ \leqslant\ k(u^2-2s) + 6\ell + 2ks+4+12u+6(u-1)\\
&\ \leqslant\ ku^2 + 6(u^2/(2s)-1) + 18u-2.
\end{align*}
Hence, 
$$u^2\ \leqslant\ (k-3)u^2 \ \leqslant\ 18u-8,$$
giving $2\leqslant u\leqslant 17$ and $4\leqslant k\leqslant 10$. For each value of $u$, $s$ is between $1$ and $u-1$. Feeding all these possible triples $(u, s, k)$ into \eqref{ep50}, we see that only the triple $(u,s, k) = (2,1,10)$ gives an integer value for $\ell$ and satisfies $2s(\ell+1)\leqslant u^2$. It follows that $q = 17$, contradicting $(q,u)\neq (17,2)$. 
\item [ii)] Case 3.2: $v = 3$. Suppose, for a contradiction, that $3s(\ell+1)\leqslant 2u^2$. We obtain from \eqref{ep50} that
\begin{align*}
    3ku^2 &\ =\ 3ks\ell + 9\ell + 3(ks+2) + 18u+9s\\
    &\ \leqslant\ k(2u^2-3s) + 9(2u^2/(3s)-1) + 3ks+6+18u+9(u-1)\\
    &\ =\ 2ku^2+6u^2/s + 27u-12.
\end{align*}
Hence, 
$$ku^2\ \leqslant\ 6u^2/s + 27u-12.$$

\begin{enumerate}
    \item[a)] If $s \geqslant 2$, then $u\geqslant 3$, and we have 
    $$u^2\ \leqslant\ (k-3)u^2\ \leqslant\ 27u-12;$$
    hence, $3\leqslant u\leqslant 26$ and $4\leqslant k\leqslant 13$. For each value of $u$, $s$ is between $2$ and $u-1$. Feeding all these possible triples $(u, s, k)$ into \eqref{ep50}, we see that none produces an integer value for $\ell$ and satisfies $3s(\ell+1)\leqslant 2u^2$. 
    \item[b)] If $s = 1$, then $(s, v) = (1,3)$, which is excluded as we state at the beginning of the present proof. 
\end{enumerate}
\end{enumerate}
\end{proof}

\subsection{Proof of Theorem \ref{fthm3}}

Plugging $(k,v) = (4j+3, 4s+8)$ into the quantity inside the floor function in \eqref{s5} to have
$$x\ :=\ \frac{(4j+3)(4(4j+3)(s+2)+1)(16(j+1)(s+2)-1)}{8j+7}.$$
\begin{claim}\label{cll5}
For $j\geqslant 6$, 
$$x-\lfloor x\rfloor \ =\ \frac{2s^2+6s+4j+8}{8j+7}.$$
\end{claim}
\begin{proof}
Since $j\geqslant s(s+1)$ and $s\geqslant 2$, 
$$0\ <\ \frac{2s^2+6s+4j+8}{8j+7}\ <\ 1.$$
Furthermore, $x - (2s^2+6s+4j+8)/(8j+7) =$
$$331+832j+512j^2+330s+832js+512j^2s+82s^2+208js^2+128j^2s^2$$
is a positive integer. This completes our proof. 
\end{proof}

From the proof of Claim \ref{cll5}, we know that the right side of \eqref{s5} is
$$R\ :=\ 332+832j+512j^2+330s+832js+512j^2s+82s^2+208js^2+128j^2s^2.$$
Then \eqref{s5} is equivalent to 
\begin{align*}
&(4j+3)(4(4j+3)(s+2)+1)(16(j+1)(s+2)-1)+4j+3+\frac{1}{4s+8}\\
&\ > \ \left(8j+7+\frac{1}{(4j+3)(4s+8)^2}\right)R.
\end{align*}
\textsc{Mathematica} simplifies the inequality to
$$64(s+2)j^2 - 8(8s^3+40s^2+51s+7)j - (48s^3+240s^2+343s+115)\ <\ 0.$$
It suffices to verify that $j$ lies between the two roots of the equation
$$64(s+2)x^2 - 8(8s^3+40s^2+51s+7)x - (48s^3+240s^2+343s+115)\ =\ 0.$$
The two roots $x_1 < x_2$ are
$$\frac{4(8s^3+40s^2+51s+7)\pm\sqrt{\Delta}}{64(s+2)},$$
where $\Delta = 16(64s^6+640s^5+2608s^4+5536s^3+6453s^2+3918s+969)$. It is easy to verify that for all $s\geqslant 2$, 
$$x_1 \ <\ s(s+1)\ \leqslant\ j\ \leqslant\ (s+1)(s+2)/2-1 \ <\ x_2.$$

\section*{Statements and Declarations}
The authors have no competing interests to declare that are relevant to the content of this article.


\ \\
\end{document}